\documentclass[review]{elsarticle}

\usepackage{lineno,hyperref}

\usepackage{amssymb,amsmath,amsthm}

\newtheorem{theorem}{Theorem}
\newtheorem{definition}{Definition}

\newtheorem{lemma}{Lemma}

\modulolinenumbers[5]

\journal{Journal of \LaTeX\ Templates}

\begin{document}

\begin{frontmatter}

\title{On a class of inverse problems for a heat equation with involution perturbation}


\author[address1]{Nasser Al-Salti}
\ead{nalsalti@squ.edu.om}

\author[mymainaddress]{Mokhtar Kirane\corref{mycorrespondingauthor}}
\cortext[mycorrespondingauthor]{Corresponding author}
\ead{mokhtar.kirane@univ-lr.fr}

\author[address2]{Berikbol T. Torebek}
\ead{torebek@math.kz}

\address[address1]{Department of Mathematics and Statistics, Sultan Qaboos University, PO Box 36, Al-Khodh, PC 123, Muscat, Oman}
\address[mymainaddress]{LaSIE, Facult\'{e} des Sciences et Technologies, Universit\'{e} de La Rochelle, Avenue Michel Cr\'{e}peau, 17000, La Rochelle, France}
\address[address2]{Institute of Mathematics and Mathematical Modeling, 050010, Pushkin st., 125, Almaty, Kazakhstan}

\begin{abstract}
A class of inverse problems for a heat equation with involution perturbation is considered using four different boundary conditions, namely, Dirichlet, Neumann, periodic and anti-periodic boundary conditions. Proved theorems on existence and uniqueness of solutions to these problems are presented. Solutions are obtained in the form of series expansion using a set of appropriate orthogonal basis for each problem. Convergence of the obtained solutions is also discussed.
\end{abstract}

\begin{keyword}
Inverse Problems, Heat Equation, Involution Perturbation
\MSC[2010] 35R30; 35K05; 39B52
\end{keyword}

\end{frontmatter}


\section{Introduction}

Differential equations with modified arguments are equations in which the unknown function and its derivatives are evaluated with modifications of  time or space variables; such equations are called in general functional differential equations. Among such equations, one can single out, equations with involutions \cite{Cabada2}. 

\begin{definition} \cite{Carleman, Wiener2} A function $\alpha(x)\not\equiv x,$ that maps a set of real numbers, $\Gamma$ onto itself and satisfies on $\Gamma$ the condition $$\alpha\left(\alpha(x)\right)=x,\quad \textrm{or}\quad \alpha^{-1}(x)=\alpha(x)$$ is called an involution on $\Gamma.$
\end{definition}
Equations containing involution are equations with an alternating deviation (at $x^* < x$ being equations with advanced, and at $x^*> x$ being equations with delay, where $x^*$ is a fixed point of the mapping $\alpha(x)$).

Equations with involutions have been studied by many researchers, for example, Ashyralyev \cite{Ashyralyev1, Ashyralyev2}, Babbage \cite{Babbage}, Przewoerska-Rolewicz \cite{D1,D2,D3,D4,D5,D6,D7}, Aftabizadeh and Col. \cite{Aftabizadeh}, Andreev \cite{Andreev1,Andreev2}, Burlutskayaa and Col. \cite{Burlutskayaa}, Gupta \cite{Gupta1,Gupta2,Gupta3}, Kirane \cite{Kirane2}, Watkins \cite{Watkins}, and Wiener \cite{Viner1,Viner2,Wiener1,Wiener2}. Spectral problems and inverse problems for equations with involutions have received a lot of attention as well, see for example, \cite{Kirane1, Kopzhassarova,Sadybekov,Sarsenbi1,Sarsenbi2}, and equations with a delay in the space variable have been the subject of many research papers, see for example, \cite{Aliev, Rus}.
Furthermore, for the equations containing transformation of the spatial variable in the diffusion term , we can cite the talk of Cabada and Tojo \cite{Cabada1}, where they gave an example that describes a concrete situation in physics: Consider a metal wire around a thin sheet of insulating material in a way that some parts overlap some others as shown in Figure \ref{fig1}.

\begin{figure}[h!]
\begin{center}
  \includegraphics[width=0.25\linewidth]{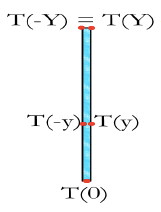}
  \ \ \caption{An application of heat equation with involution}
  \label{fig1}
\end{center}
\end{figure}

Assuming that the position $y = 0$ is the lowest of the wire, and the insulation goes up to the left at $-Y$ and to the right up to $Y.$ For the proximity of two sections of wires they added the third term with modifications on the spatial variable to the right-hand side of the heat equation with respect to the wire: $$\frac{\partial T}{\partial t}(y,t)=a \frac{\partial^2 T}{\partial y^2}(y,t)+b \frac{\partial^2 T}{\partial y^2}(-y,t).$$
Such equations have also a purely theoretical value. For general facts about partial functional differential equations and for properties of equations with involutions in particular, we refer the reader to the books of Skubachevskii \cite{Skubachevskii}, Wu \cite{Wu} and Cabada and Tojo \cite{Cabada2}.

In this paper, we consider inverse problems for a heat equation with involution using four different boundary conditions. We seek formal solutions to these problems in a form of series expansions using orthogonal basis obtained by separation of variables and we also examine the convergence of the obtained series solutions. The main results on existence and uniqueness are formulated in four theorems in the last section of this paper along with an illustrating example.

Concerning inverse problems for heat equations, some recent works have been implemented by Kaliev \cite{Kaliev1,Kaliev2}, Sadybekov \cite{Orazov1,Orazov2}, Kirane \cite{Furati, Kirane3}.

\section{Statements of Problems}
Consider the heat equation

\begin{equation}\label{2.1}u_t \left( {x,t} \right) - u_{xx} \left( {x,t} \right) +
\varepsilon u_{xx} \left( { - x,t} \right) = f\left( x \right), \quad
\left( {x,t} \right) \in \Omega,
\end{equation}
where, $\varepsilon$ is a nonzero real number such that $\left| \varepsilon \right| < 1$ and  $\Omega$ is a rectangular domain given by $\Omega = \left\{ {- \pi  < x < \pi, \, 0 < t < T} \right\}$. Our aim is to find a regular solution to the following four inverse problems: \\

{\bf IP1: Inverse Problem with Dirichlet Boundary Conditions.} \\ Find a pair of functions $u\left( {x,t} \right)$ and $f\left( x \right)$ in the domain $\Omega $ satisfying equation (\ref{2.1}) and the conditions
\begin{equation}\label{2.2}
u\left({x,0} \right) = \varphi \left( x \right), \quad u\left( {x,T} \right) =
\psi \left( x \right), \quad x \in \left[ { - \pi ,\pi } \right],
\end{equation}
and the homogeneous Dirichlet boundary conditions
\begin{equation}\label{2.3}
u\left( { - \pi ,t} \right) = 0, \quad u\left( {\pi ,t} \right) = 0, \quad t \in \left[ {0,T}
\right],\end{equation}
where $\varphi \left( x \right)$ and $\psi
\left( x \right)$ are given, sufficiently smooth functions.\\

{\bf IP2: Inverse Problem with Neumann Boundary Conditions.} \\ Find a pair of functions $u\left( {x,t} \right)$ and $f\left( x \right)$ in the domain $\Omega $ satisfying equation (\ref{2.1}), conditions (\ref{2.2}) and the homogeneous Neumann boundary conditions
\begin{equation}\label{2.4}
u_x \left( { - \pi ,t} \right) = 0, \quad u_x \left( {\pi ,t} \right) = 0, \quad t \in \left[ {0,T}\right].
\end{equation}

{\bf IP3: Inverse Problem with Periodic Boundary Conditions.} \\ Find a pair of functions $u\left( {x,t} \right)$ and $f\left( x \right)$ in the domain $\Omega $ satisfying equation (\ref{2.1}), conditions (\ref{2.2}) and the periodic boundary conditions
\begin{equation}\label{2.5}
u\left( { - \pi ,t} \right) = u\left(
{\pi ,t} \right), \quad u_x \left( { - \pi ,t} \right) = u_x \left( {\pi
,t} \right), \quad t \in \left[ {0,T} \right].
\end{equation}

{\bf IP4: Inverse Problem with Anti-Periodic Boundary Conditions.} \\ Find a pair of functions $u\left( {x,t} \right)$ and $f\left( x \right)$ in the domain $\Omega $ satisfying equation (\ref{2.1}), conditions (\ref{2.2}) and the anti-periodic boundary conditions
\begin{equation}\label{2.6}
u\left( { - \pi ,t} \right) = - u\left(
{\pi ,t} \right), \quad u_x \left( { - \pi ,t} \right) = - u_x \left( {\pi
,t} \right), \quad t \in \left[ {0,T} \right].
\end{equation}

By a regular solution of problems IP1, IP2, IP3 and IP4, we mean a pair of
functions $u\left( {x,t} \right)$ and $f\left( x \right)$ of the
class $u\left( {x,t} \right) \in C_{x,t}^{2,1} \left( { \Omega
} \right),$ $f\left( x \right) \in C\left[ { - \pi ,\pi }
\right].$

\section{Solution Method}
Here we seek a solution to problems IP1, IP2, IP3 and IP4 in a form of series expansion using a set of functions that form orthogonal basis in $L_2(-\pi, \pi)$. To find the appropriate set of functions for each problem, we shall solve the homogeneous equation corresponding to equation (\ref{2.1}) along with the associated boundary conditions using separation of variables.

\subsection{Spectral Problems}
Separation of variables leads to the following spectral problems for IP1, IP2, IP3 and IP4, respectively,
\begin{equation}\label{SP IP1}
X''(x)-\epsilon X''(-x)+\lambda X(x)=0,\quad \quad X(-\pi)=X(\pi)=0,
\end{equation}
\begin{equation}\label{SP IP2}
X''(x)-\epsilon X''(-x)+\lambda X(x)=0,\quad X'(-\pi)=X'(\pi)=0,
\end{equation}
\begin{equation}\label{SP IP3}
X''(x)-\epsilon X''(-x)+\lambda X(x)=0,\quad X(-\pi)=X(\pi), \, X'(-\pi)=X'(\pi),
\end{equation}
\begin{equation}\label{SP IP4}
X''(x)-\epsilon X''(-x)+\lambda X(x)=0,\, X(-\pi)=-X(\pi), \, X'(-\pi)=-X'(\pi).
\end{equation}
The eigenvalue problems (\ref{SP IP1}) - (\ref{SP IP4}) are self-adjoint and hence they have real eigenvalues and their eigenfunctions form a complete orthogonal basis in $L_2 \left( { - \pi ,\pi } \right)$ \cite{Naim}. Their eigenvalues are, respectively, given by
\begin{equation*} \label{eigenvalue IP1}
\lambda _{1k}  = \left( {1 - \varepsilon } \right)\left( {k + \frac{1}{2}} \right)^2, \, k \in \mathbb{N}\cup\left\{0\right\}, \quad \lambda _{2k}  = \left( {1 + \varepsilon }\right)k^2, \, k \in \mathbb{N}, \tag{7.a}
\end{equation*}
\begin{equation*}\label{eigenvalue IP2}
\lambda _{1k}  = \left( {1 - \varepsilon } \right)k^2, \quad \lambda _{2k}  = \left( {1 + \varepsilon }
\right)\left( {k + \frac{1}{2}} \right)^2, \quad k \in \mathbb{N}\cup\left\{0\right\}, \tag{8.a}
\end{equation*}
\begin{equation*}\label{eigenvalue IP3}
\lambda _{1k}  = \left( {1 - \varepsilon } \right)k^2, \, k \in \mathbb{N}\cup\left\{0\right\}, \quad  \quad \lambda _{2k}  = \left( {1 + \varepsilon }\right)k^2, \, k \in \mathbb{N}, \tag{9.a}
\end{equation*}
\begin{equation*}\label{eigenvalue IP4}
\lambda _{1k}  = \left( {1 + \varepsilon } \right)\left( {k + \frac{1}{2}} \right)^2, \quad  \lambda _{2k}  = \left( {1 - \varepsilon }\right)\left( {k + \frac{1}{2}} \right)^2, \, k \in \mathbb{N}\cup\left\{0\right\}, \tag{10.a}
\end{equation*}
and the corresponding eigenfunctions are given by
\begin{equation*}\label{basis IP1}
X_{1k} = \cos \left( {k + \frac{1}{2}} \right)x,\,\, k \in \mathbb{N}\cup\left\{0\right\}, \quad X_{2k} = \sin kx, \, \, k \in \mathbb{N}, \tag{7.b}
\end{equation*}
\begin{equation*}\label{basis IP2}
X_0=1, \,  X_{1k} = \cos kx,  \, k \in \mathbb{N}, \, \, \, X_{2k} = \sin \left( {k + \frac{1}{2}} \right)x,\, k \in \mathbb{N}\cup\left\{0\right\},\tag{8.b}
\end{equation*}
\begin{equation*}\label{basis IP3}
X_0=1, \, \, \quad X_{1k} = \cos kx,\,\, \quad X_{2k} = \sin kx, \, \, \quad k \in \mathbb{N}. \tag{9.b}
\end{equation*}
\begin{equation*}\label{basis IP4}
X_{1k} =\sin \left( {k + \frac{1}{2}} \right)x \quad X_{2k} = \cos \left( {k + \frac{1}{2}} \right)x  , \, \, \quad k \in \mathbb{N}\cup\left\{0\right\}. \tag{10.b}
\end{equation*}

\begin{lemma}\label{l1}
The systems of functions (\ref{basis IP1}) - (\ref{basis IP4}) are complete and orthogonal in $L_2\left( { - \pi ,\pi } \right).$
\end{lemma}
\begin{proof}
Here we present the proof for the system of functions (\ref{basis IP1}). The orthogonality follows from the direct calculations:
$$ \int_{-\pi}^{\pi} X_{1n}X_{2m}\, dx=0, \quad  n \in \mathbb{N}\cup\left\{0\right\}, \, m \in \mathbb{N},$$
and
$$ \int_{-\pi}^{\pi} X_{in}X_{im}\, dx=0, \quad m \ne n, \, i=1,2.$$
Hence, it only remains  to prove the completeness of the system in $L_2(-\pi, \pi)$, i.e., we need to show that if
\begin{equation} \label{complete1}
\int_{-\pi}^{\pi}f(x)\cos\left(k+\frac{1}{2}\right)x \, dx=0, \quad k \in \mathbb{N}\cup\left\{0\right\},
\end{equation}
and
\begin{equation}\label{complete2}
\int_{-\pi}^{\pi}f(x)\sin kx \,dx=0, \quad k\in N,
\end{equation}
then $f(x) \equiv 0 $ in $(-\pi, \pi)$.
To show this, we are going to use the fact that $\left\{\cos \left(k+\frac{1}{2}\right)x\right\}_{k \in \mathbb{N}\cup\left\{0\right\}}$ and $\left\{\sin kx\right\}_{k\in N}$ are complete in $L_2(0, \pi)$, see \cite{Moiseev} for example. Now, suppose that the equation (\ref{complete1}) holds. We then have
$$0=\int_{-\pi}^{\pi}f(x)\cos\left(k+\frac{1}{2}\right)x \,dx=\int_{0}^{\pi}\left(f(x)+f(-x)\right)\cos\left(k+\frac{1}{2}\right)x\,dx.$$
Hence, by the completeness of the system
$\left\{\cos \left(k+\frac{1}{2}\right)x\right\}_{k \in \mathbb{N}\cup\left\{0\right\}}$ in $L_2(0, \pi)$, we have $f(x)=-f(-x), \, -\pi<x<\pi$.
Similarly, if equation (\ref{complete2}) holds, we have
$$0=\int_{-\pi}^{\pi}f(x)\sin kx \,dx=\int_{0}^{\pi}\left(f(x)-f(-x)\right)\sin kx \,dx.$$
Then, by the completeness of the system $\left\{\sin kx\right\}_{k\in N}$ in $L_2(0, \pi)$, we have $f(x)=f(-x), \, -\pi<x<\pi.$ Therefore, we must have $f(x) \equiv 0$ in $(-\pi, \pi).$
Completeness and orthogonality of the systems of functions (\ref{basis IP2}) - (\ref{basis IP4}) can be proved similarly.
\end{proof}

Since each one of the systems of eigenfunctions (\ref{basis IP1}) - (\ref{basis IP4}) is complete and forms a basis in $L_2\left(-\pi, \pi\right)$, the solution pair $u(x,t)$ and $f(x)$ of each inverse problem can be expressed  in a form of series expansion using the appropriate set of eigenfunctions.

\subsection{Existence of Solutions}
Here, we give a full proof of existence of solution to the Inverse Problem IP1. Existence of solutions to the other three problems can be proved similarly. Using the orthogonal system (\ref{basis IP1}), the functions $u\left(
{x,t} \right)$ and $f\left( x \right)$ can be represented as follows
\begin{equation}\label{6.1}
u\left( {x,t} \right) = \sum\limits_{k
= 0}^\infty  {u_{1k} \left( t \right)\cos \left( {k + \frac{1}{2}}
\right)x}  + \sum\limits_{k = 1}^\infty {u_{2k} \left( t \right)\sin
kx},\end{equation}
\begin{equation}\label{6.2}
f\left( x \right) =
\sum\limits_{k = 0}^\infty {f_{1k} \cos \left( {k + \frac{1}{2}}
\right)x}  + \sum\limits_{k = 1}^\infty {f_{2k} \sin kx},
\end{equation}
where the coefficients $u_{1k} \left( t \right),u_{2k}
\left( t \right), f_{1k} ,f_{2k}$ are unknown. Substituting (\ref{6.1}) and
(\ref{6.2}) into equation (\ref{2.1}), we obtain the following equations relating the functions $u_{1k} \left( t \right),$ $u_{2k} \left( t \right)$ and the constants $f_{1k}, f_{2k}$:
\begin{equation} \label{u1k_eqn}
u'_{1k} \left( t \right) + \left({1 - \varepsilon } \right)\left( {k + \frac{1}{2}} \right)^2 u_{1k}\left( t \right) = f_{1k} ,
\end{equation}
\begin{equation} \label{u2k_eqn}
u'_{2k} \left( t \right) + \left( {1 + \varepsilon } \right)k^2 \, u_{2k} \left( t \right) = f_{2k}.
\end{equation}
Solving these equations we obtain
$$u_{1k} \left( t \right) = \frac{{f_{1k} }}{{\left( {1 - \varepsilon } \right)\left( {k + \frac{1}{2}} \right)^2 }} + C_{1k} e^{ - \left( {1 - \varepsilon } \right)\left( {k + \frac{1}{2}} \right)^2 t},$$
$$u_{2k} \left( t \right) = \frac{{f_{2k} }}{{\left( {1 + \varepsilon } \right)k^2 }} + C_{2k} e^{ - \left( {1 + \varepsilon } \right)k^2 t},$$
where the unknown constants $C_{1k} ,$ $C_{2k} ,$ $f_{1k},$ $f_{2k}$ are to be determined using the conditions in (\ref{2.2}). Let $\varphi
_{ik} ,\psi _{ik} ,i = 1,2$ be the coefficients of the series expansions
of $\varphi \left( x \right)$ and $\psi \left( x \right)$, respectively, i.e.,
$$\varphi _{1k}  =  \frac{1}{\pi}\int\limits_{ - \pi }^\pi  {\varphi \left( x \right)\cos \left( {k + \frac{1}{2}} \right)x \, dx} , \quad \varphi
_{2k}  =  \frac{1}{\pi}\int\limits_{ - \pi }^\pi  {\varphi \left( x \right)\sin kx \, dx},$$
$$\psi _{1k}  =  \frac{1}{\pi}\int\limits_{ - \pi }^\pi  {\psi \left( x \right)\cos \left( {k + \frac{1}{2}} \right)x \, dx}, \quad \psi _{2k} =
 \frac{1}{\pi}\int\limits_{ - \pi }^\pi  {\psi \left( x \right)\sin kx \, dx}.$$
Then, the two conditions in (\ref{2.2}) leads to
$$\frac{{f_{1k}
}}{{\left( {1 - \varepsilon } \right)\left( {k + \frac{1}{2}}
\right)^2 }} + C_{1k}  = \varphi _{1k}, \quad \frac{{f_{1k} }}{{\left( {1 - \varepsilon } \right)\left( {k +
\frac{1}{2}} \right)^2 }} + C_{1k} e^{ - \left( {1 - \varepsilon }
\right)\left( {k + \frac{1}{2}} \right)^2 T}  = \psi _{1k}, $$

$$\frac{{f_{2k}
}}{{\left( {1 + \varepsilon } \right)k^2 }} + C_{2k}  = \varphi
_{2k}, \quad \frac{{f_{2k}}}{{\left( {1 +
\varepsilon } \right)k^2 }} + C_{2k} e^{ - \left( {1 + \varepsilon
} \right)k^2 T}  = \psi _{2k}.$$
Solving these set of algebraic equations, we get
$$C_{1k}  = \frac{{\varphi _{1k}  - \psi _{1k} }}{{1 - e^{ - \left(
{1 - \varepsilon } \right)\left( {k + \frac{1}{2}} \right)^2 T}
}}, \quad f_{1k}  = \left( {1
- \varepsilon } \right)\left( {k + \frac{1}{2}} \right)^2 \left( \varphi
_{1k}  - C_{1k} \right),$$

$$C_{2k}  = \frac{{\varphi _{2k}  - \psi _{2k} }}{{1 - e^{
- \left( {1 + \varepsilon } \right)k^2 T} }}, \quad  \quad  \quad f_{2k}  = \left( {1 + \varepsilon } \right)k^2  \left(\varphi _{2k}  - C_{2k} \right).$$
Now, substituting $u_{1k}\left( t \right),$ $u_{2k} \left( t \right),$ $f_{1k},$ $f_{2k}$ into (\ref{6.1}) and (\ref{6.2}) we get
\begin{eqnarray*}
u\left( {x,t} \right) =
\varphi \left( x \right) &+& \sum\limits_{k = 0}^\infty  {C_{1k}
\left( {e^{ - \left( {1 - \varepsilon } \right)\left( {k +
\frac{1}{2}} \right)^2 t}  - 1} \right)\cos \left( {k +
\frac{1}{2}} \right)x} \\
 &+& \sum\limits_{k = 1}^\infty {C_{2k}
\left( {e^{ - \left( {1 + \varepsilon } \right)k^2 t}  - 1}
\right)\sin kx},
\end{eqnarray*}
and
\begin{eqnarray*}
f\left( x
\right) = -\varphi'' \left( x \right) +\varepsilon\varphi'' \left( -x \right)
&-& \sum\limits_{k = 0}^\infty {\left( 1 - \varepsilon \right)  \left( {k + \frac{1}{2}} \right)^2 C_{1k}  \cos \left( {k + \frac{1}{2}} \right)x} \\
&-& \sum\limits_{k = 1}^\infty {\left( 1 +\varepsilon \right) k ^2 \, C_{2k} \sin kx} .
\end{eqnarray*}
Note that for $f\left( x \right) \in C\left[ { - \pi ,\pi }
\right]$, it is required that $\varphi(x) \in C^2\left[ { - \pi ,\pi }\right]$.

\subsection{Convergence of Series}
In order to justify that the obtained formal solution is indeed a true solution, we need to show that the series appeared in $u(x,t)$ and $f(x)$ as well as the corresponding series representations of $u_{xx}(x,t)$ and $u_t(x,t)$ converge uniformly in $\Omega$. For this purpose, let
$$\varphi ^{\left( i \right)} \left( { - \pi } \right) = \varphi ^{\left( i \right)} \left( \pi  \right) = 0, \quad i = 0,2,$$
$$\psi ^{\left( i \right)} \left( { - \pi } \right) = \psi ^{\left( i \right)} \left( \pi  \right) = 0, \quad i = 0,2.$$
Hence, on integration by parts, $C_{1k}$  and   $C_{2k}$ can now be rewritten as
$$C_{1k}  = \frac{{\varphi _{2k}^{\left( 3 \right)}  - \psi
_{2k}^{\left( 3 \right)} }}{{\left( {1 - e^{ - \left( {1 -
\varepsilon } \right)\left( {k + \frac{1}{2}} \right)^2 T} }
\right)\left( {k + \frac{1}{2}} \right)^3 }}, \quad C_{2k} = -\frac{{\varphi _{1k}^{\left( 3 \right)}  - \psi
_{1k}^{\left( 3 \right)} }}{{\left( {1 - e^{ - \left( {1 +
\varepsilon } \right)k^2 T} } \right)k^3 }}.$$
where,
$$\varphi _{1k} ^{(3)} =  \frac{1}{\pi}\int\limits_{ - \pi }^\pi  {\varphi''' \left( x \right)\cos kx \, dx} , \quad \varphi
_{2k}^{(3)}  =  \frac{1}{\pi}\int\limits_{ - \pi }^\pi  {\varphi''' \left( x \right)\sin \left( {k + \frac{1}{2}} \right) x  \, dx},$$
$$\psi _{1k}^{(3)}  =  \frac{1}{\pi}\int\limits_{ - \pi }^\pi  {\psi''' \left( x \right)\cos kx \,  dx}, \quad \psi _{2k}^{(3)} =
 \frac{1}{\pi}\int\limits_{ - \pi }^\pi  {\psi''' \left( x \right)\sin \left( {k + \frac{1}{2}} \right) x \, dx}.$$
Hence, the series representation of $u\left({x,t} \right)$ and $f(x)$ can be expressed as
\begin{eqnarray*}
u\left(
{x,t} \right) = \varphi \left( x \right) &+& \sum\limits_{k = 1}^\infty  {\frac{{{1-e^{ - \left( {1 + \varepsilon } \right)k^2 t}} }}{{{1 - e^{ - \left( {1 + \varepsilon } \right)k^2 T} } }}  \left( \frac{ {\varphi _{1k}^{\left( 3 \right)}  - \psi_{1k}^{\left( 3 \right)} }}{k^3} \right) \sin \, kx}\\
&-& \sum\limits_{k = 0}^\infty {\frac{{{1-e^{ - \left( {1 - \varepsilon }\right)\left( {k + \frac{1}{2}} \right)^2 t}} }}{{ {1 - e^{ - \left( {1 -
\varepsilon } \right)\left( {k + \frac{1}{2}} \right)^2 T} }}}\left( \frac{ {\varphi_{2k}^{\left( 3 \right)}  - \psi _{2k}^{\left( 3 \right)} }}{\left( {k + \frac{1}{2}} \right)^3} \right) \cos \left( {k + \frac{1}{2}} \right)x} ,
\end{eqnarray*}
 and

\begin{eqnarray*}
f\left( x
\right) = -\varphi'' \left( x \right) &+&\varepsilon\varphi'' \left( -x \right) +\sum\limits_{k = 1}^\infty  {\frac{1 + \varepsilon}{k} \left(\frac{\varphi _{1k}^{\left( 3 \right)}  - \psi
_{1k}^{\left(3 \right)} }{{{1 - e^{ - \left( {1+ \varepsilon } \right)k^2 T} }}} \right)\sin kx} \\
&-& \sum\limits_{k = 0}^\infty {\frac{\left( {1 - \varepsilon } \right)}{\left( {k + \frac{1}{2}} \right)} \left(\frac{{{\varphi
_{2k}^{\left( 3 \right)}  - \psi _{2k}^{\left( 3 \right)} }
}}{{{1 - e^{ - \left( {1 - \varepsilon }
\right)\left( {k + \frac{1}{2}} \right)^2 T} } }}\right)\cos \left( {k + \frac{1}{2}} \right)x}.
\end{eqnarray*}
For convergence, we then have the following estimates for $u\left( {x,t} \right)$ and $f \left(x \right)$
$$\left|
{u\left( {x,t} \right)} \right| \le \left| {\varphi \left( x
\right)} \right| +c \sum\limits_{k = 1}^\infty  {\frac{{\left|
{\varphi _{1k}^{\left( 3 \right)} } \right| + \left| {\psi
_{1k}^{\left( 3 \right)} } \right|}}{{k^3 }}}+ c\sum\limits_{k = 0}^\infty
{\frac{{\left| {\varphi _{2k}^{\left( 3 \right)} } \right| +
\left| {\psi _{2k}^{\left( 3 \right)} } \right|}}{{\left( {k + \frac{1}{2}} \right)^3 }}}$$
and
\begin{eqnarray*}
\left| {f\left( x \right)} \right| \le
\left| {\varphi'' \left( x \right)} \right| &+&\left| {\varphi'' \left( -x \right)} \right|+ c\sum\limits_{k =
1}^\infty {\left( \left| {\varphi _{1k}^{\left( 3 \right)} }
\right|^2 + \left| {\psi _{1k}^{\left( 3 \right)} }
\right|^2 + \frac{2}{k^2}\right)} \\
&+& c\sum\limits_{k =
0}^\infty {\left( \left| {\varphi _{2k}^{\left( 3 \right)} }
\right|^2 + \left| {\psi _{2k}^{\left( 3 \right)} }
\right|^2 + \frac{2}{\left( {k + \frac{1}{2}} \right)^2}\right)},
\end{eqnarray*}
for some positive constant $c$. Here, for the estimate of $f(x)$, we have used the inequality $2ab \le a^2 +b^2$. The convergence of the series in the estimate of $u(x,t)$ is clearly achieved if $\varphi^{(3)}_{ik} ,\psi^{(3)} _{ik} ,i = 1,2$ are finite. This can be ensured by assuming that $\varphi'''(x)$ and $\psi'''(x) \in L_2(-\pi, \pi)$. Furthermore, by Bessel inequality for trigonometric series, the following
series converge:
$$\sum\limits_{k =1}^\infty  {\left| {\varphi _{ik}^{\left(3 \right)} } \right|^2 \le } C\left\| {\varphi''' \left( x \right)}
\right\|_{L_2 \left( { - \pi ,\pi } \right)}^2 , \quad i =1,2, $$
$$\sum\limits_{k =1}^\infty  {\left| {\psi _{ik}^{\left(3 \right)} } \right|^2 \le } C\left\| {\psi''' \left( x \right)}
\right\|_{L_2 \left( { - \pi ,\pi } \right)}^2 , \quad i =1,2. $$
Therefore, by the Weierstrass
M-test (see.\cite{Knopp}), the series representations of $u(x,t)$ and $f(x)$
converge absolutely and uniformly in the region $\Omega .$ The convergence of the series representations of $u_{xx}(x,t)$ and $u_t(x,t)$ which are obtained by term-wise differentiation of the series representation of $u(x,t)$ can be shown is a similar way.

\subsection{Uniqueness of Solution}
Suppose that there are two solution sets $\left\{ {u_1 \left( {x,t} \right),f_1 \left( x \right)} \right\}$ and $\left\{ {u_2 \left( {x,t} \right), f_2\left( x \right)} \right\}$ to the Inverse Problem IP1. Denote $$u\left( {x,t} \right) = u_1 \left( {x,t} \right) - u_2 \left( {x,t} \right),$$
and
$$f\left( x \right) = f_1 \left( x \right) - f_2 \left( x
\right).$$ Then, the functions $u\left( {x,t} \right)$ and $f\left(
x \right)$ clearly satisfy equation (\ref{2.1}), the boundary conditions in (\ref{2.3}) and the homogeneous conditions
\begin{equation} \label{H_IC}
u\left({x,0} \right) = 0, \quad u\left( {x,T} \right) =
0, \quad x \in \left[ { - \pi ,\pi } \right]
\end{equation}
Let us now introduce the following
\begin{equation}\label{u1k}
u_{1k} \left(t \right) = \frac{1}{\pi}\int\limits_{ - \pi }^\pi {u\left( {x,t} \right)\cos \left(k+ \frac{1}{2}\right) x\,dx}, \quad k \in \mathbb{N}\cup\left\{0\right\},
\end{equation}
\begin{equation}\label{u2k}
u_{2k}
\left( t \right) = \frac{1}{\pi}\int\limits_{ - \pi }^\pi {u\left( {x,t}
\right)\sin kx \, dx}, \quad k \in N,
\end{equation}
\begin{equation}\label{f1k}
f_{1k} = \frac{1}{\pi}\int\limits_{ - \pi }^\pi {f\left( x \right)\cos \left(k+ \frac{1}{2}\right)x\, dx},  \quad k \in \mathbb{N}\cup\left\{0\right\},
\end{equation}
\begin{equation}\label{f2k}
f_{2k} = \frac{1}{\pi}\int\limits_{ - \pi }^\pi
{f\left( x \right)\sin kx \, dx}, \quad k \in N.
\end{equation}
Note that the homogeneous conditions in (\ref{H_IC}) lead to
\begin{equation}
u_{ik}(0)= u_{ik}(T)=0, \quad i=1,2,
\end{equation}
and differentiating equation (\ref{u1k}) gives
$$u'_{1k} \left( t
\right) = \frac{1}{\pi}\int\limits_{ - \pi }^\pi {\left( {u_{xx} \left( {x,t} \right) -
\varepsilon u_{xx} \left( { - x,t} \right)} \right)\cos \left(k+ \frac{1}{2}\right)x\, dx}  + f_{1k},$$
which on integrating by parts and using the conditions in (\ref{2.2}) reduces to
$$u'_{1k} \left( t
\right) = (\varepsilon-1) \left(k+\frac{1}{2}\right)^2 u_{1k}  + f_{1k}.$$
One can then easily show that this equation together with the conditions $u_{1k}(0)=u_{1k}(T)=0$ imply that
 $$f_{1k}  = 0, \quad u_{1k} \left( t \right) \equiv 0.$$
Similarly, for $u_{2k}$ and $f_{2k}$ as given in (\ref{u2k}) and (\ref{f2k}), respectively, one can show that
 $$f_{2k}  = 0, \quad u_{2k} \left( t \right) \equiv 0.$$
Therefore, due to the completeness of the system of eigenfunctions (\ref{basis IP1}) in $L_2 \left({ - \pi ,\pi } \right)$, we must have
 $$f\left( t \right) \equiv 0, \quad u\left(
{x,t} \right) \equiv 0,  \quad \left( {x,t} \right) \in \bar{\Omega} .$$
This ends the proof of uniqueness of solution to the Inverse Problem IP1. Uniqueness of solutions to the Inverse Problems IP2, IP3 and IP4 can be proved in a similar way.

\section{Main Results and Example Solution}

\subsection{Main Results} 
The main results for the Inverse Problems IP1, IP2, IP3 and IP4 can be summarized in the following theorems:

\begin{theorem}\label{th1} Let $\varphi \left( x \right),\psi \left( x \right) \in C^2 \left[ { - \pi ,\pi } \right]$, $\varphi'''(x)$, $\psi'''(x) \in L_2(-\pi, \pi)$ and $\varphi ^{\left( i\right)} \left( { \pm \pi } \right) = \psi ^{\left( i \right)} \left( { \pm \pi } \right) = 0,i = 0,2.$ Then, a unique solution to
the Inverse Problem IP1 exists and it can be written in the form
\begin{eqnarray*}
u\left(
{x,t} \right) = \varphi \left( x \right) &+& \sum\limits_{k = 1}^\infty  {\frac{{{1-e^{ - \left( {1 + \varepsilon } \right)k^2 t}} }}{{{1 - e^{ - \left( {1 + \varepsilon } \right)k^2 T} } }}  \left( \frac{ {\varphi _{1k}^{\left( 3 \right)}  - \psi_{1k}^{\left( 3 \right)} }}{k^3} \right) \sin \, kx}\\
&-& \sum\limits_{k = 0}^\infty {\frac{{{1-e^{ - \left( {1 - \varepsilon }\right)\left( {k + \frac{1}{2}} \right)^2 t}} }}{{ {1 - e^{ - \left( {1 -
\varepsilon } \right)\left( {k + \frac{1}{2}} \right)^2 T} }}}\left( \frac{ {\varphi_{2k}^{\left( 3 \right)}  - \psi _{2k}^{\left( 3 \right)} }}{\left( {k + \frac{1}{2}} \right)^3} \right) \cos \left( {k + \frac{1}{2}} \right)x} ,
\end{eqnarray*}
\begin{eqnarray*}
f\left( x
\right) = -\varphi'' \left( x \right) &+&\varepsilon\varphi'' \left( -x \right) +\sum\limits_{k = 1}^\infty  {\frac{1 + \varepsilon}{k} \left(\frac{\varphi _{1k}^{\left( 3 \right)}  - \psi
_{1k}^{\left(3 \right)} }{{{1 - e^{ - \left( {1+ \varepsilon } \right)k^2 T} }}} \right)\sin kx} \\
&-& \sum\limits_{k = 0}^\infty {\frac{\left( {1 - \varepsilon } \right)}{\left( {k + \frac{1}{2}} \right)} \left(\frac{{{\varphi
_{2k}^{\left( 3 \right)}  - \psi _{2k}^{\left( 3 \right)} }
}}{{{1 - e^{ - \left( {1 - \varepsilon }
\right)\left( {k + \frac{1}{2}} \right)^2 T} } }}\right)\cos \left( {k + \frac{1}{2}} \right)x},
\end{eqnarray*}
where
$$\varphi _{1k} ^{(3)} =  \frac{1}{\pi}\int\limits_{ - \pi }^\pi  {\varphi''' \left( x \right)\cos kx \, dx} , \quad \varphi
_{2k}^{(3)}  =  \frac{1}{\pi}\int\limits_{ - \pi }^\pi  {\varphi''' \left( x \right)\sin \left( {k + \frac{1}{2}} \right) x  \, dx},$$
$$\psi _{1k}^{(3)}  =  \frac{1}{\pi}\int\limits_{ - \pi }^\pi  {\psi''' \left( x \right)\cos kx \,  dx}, \quad \psi _{2k}^{(3)} =
 \frac{1}{\pi}\int\limits_{ - \pi }^\pi  {\psi''' \left( x \right)\sin \left( {k + \frac{1}{2}} \right) x \, dx}.$$
\end{theorem}
\begin{theorem}\label{th2} Let $\varphi \left( x \right),\psi \left( x \right) \in C^2 \left[ { - \pi ,\pi } \right]$, $\varphi'''(x)$, $\psi'''(x) \in L_2(-\pi, \pi)$  and $\varphi' \left( { \pm \pi }\right) = \psi '\left( { \pm \pi } \right) =
0.$ Then a unique solution to the Inverse Problem IP2 exists and it can be written in the form
\begin{eqnarray*}
u\left( {x,t} \right) =
\varphi \left( x \right) &+&\frac{t}{T}(\psi_0-\varphi_0)+ \sum\limits_{k = 1}^\infty {\frac{{\left( {1 - e^{ - \left( {1 - \varepsilon } \right)k^2 t}
} \right)\left( {\psi _{2k}^{\left( 3 \right)}
- \varphi _{2k}^{\left( 3 \right)} } \right)}}{{\left( {1 - e^{ - \left( {1 - \varepsilon }\right)k^2 T} } \right)k^3 }}\cos kx} \\
&-&\sum\limits_{k = 0}^\infty  {\frac{{\left( {1 - e^{ - \left( {1 +
\varepsilon } \right)\left( {k + \frac{1}{2}} \right)^2 t} }
\right)\left({\psi _{1k}^{\left( 3 \right)}  - \varphi _{1k}^{\left( 3 \right)}} \right) }}{{\left( {1 - e^{ -\left( {1 + \varepsilon } \right)\left( {k + \frac{1}{2}}\right)^2 T} } \right)\left( {k + \frac{1}{2}} \right)^3 }}\sin \left( {k + \frac{1}{2}} \right)x} ,
\end{eqnarray*}
\begin{eqnarray*}
f\left( x \right) = -\varphi'' \left( x \right) &+& \varepsilon \varphi'' \left( -x \right)+ \frac{\psi_0-\varphi_0}{T} -\sum\limits_{k = 1}^\infty  {\frac{{\left( {1 - \varepsilon } \right)\left( {\varphi _{2k}^{\left( 3 \right)}  - \psi
_{2k}^{\left( 3 \right)} } \right)}}{{k \left( {1 - e^{ - \left( {1
- \varepsilon } \right)k^2 T} } \right) }}\cos kx} \\
&+&
\sum\limits_{k = 0}^\infty  {\frac{{\left( {1 + \varepsilon }
\right)\left( {\varphi _{1k}^{\left( 3 \right)}  - \psi
_{1k}^{\left( 3 \right)} } \right)}}{{\left( {k + \frac{1}{2}} \right)\left( {1 - e^{ - \left( {1+ \varepsilon } \right)\left( {k + \frac{1}{2}} \right)^2 T} }\right) }}\sin \left( {k +\frac{1}{2}} \right)x} ,
\end{eqnarray*}
where
$$\varphi _{0}  =
\frac{1}{2\pi}\int\limits_{ - \pi }^\pi  {\varphi \left( x \right)dx}, \quad\psi _{0}  =
\frac{1}{2\pi}\int\limits_{ - \pi }^\pi  {\psi \left( x \right)dx},$$
$$ \varphi _{1k}^{(3)}  = \frac{1}{\pi}\int\limits_{ -
\pi }^\pi  {\varphi''' \left( x \right)\cos \left( {k +
\frac{1}{2}} \right)x \,dx}, \quad \varphi _{2k}^{(3)} =
\frac{1}{\pi}\int\limits_{ - \pi }^\pi {\varphi''' \left( x \right)\sin kx \, dx},$$
 $$\psi _{1k}^{(3)}  =  \frac{1}{\pi}\int\limits_{ - \pi }^\pi  {\psi'''\left( x \right)\cos \left( {k +\frac{1}{2}} \right)x\, dx}, \quad \psi _{2k}^{(3)}  =
 \frac{1}{\pi}\int\limits_{ - \pi }^\pi  {\psi''' \left( x \right)\sin kx \,dx}.$$
\end{theorem}
\begin{theorem}\label{th3} Let $\varphi \left( x \right),\psi \left( x \right) \in C^2 \left[ { - \pi ,\pi } \right]$, $\varphi'''(x)$, $\psi'''(x) \in L_2(-\pi, \pi)$ and $\varphi ^{\left( i \right)} \left( { - \pi } \right) = \varphi ^{\left( i \right)} \left( \pi \right),$ $\psi ^{\left( i \right)}
\left( { - \pi } \right) = \psi ^{\left( i \right)} \left( \pi \right),$ $i = 0,1,2.$ Then, a unique solution to the Inverse Problem IP3
exists and it can be written in the form
\begin{eqnarray*}
u\left({x,t} \right) = \varphi \left( x \right) &+& \frac{t}{T}(\psi_0-\varphi_0) - \sum\limits_{k =
1}^\infty  {\frac{{\left( {1 - e^{ - \left( {1 - \varepsilon }
\right)k^2 t} } \right) \left( {\varphi
_{2k}^{\left( 3 \right)}  - \psi _{2k}^{\left( 3 \right)} }
\right)}}{{\left( {1 - e^{ - \left( {1 -
\varepsilon } \right)k^2 T} } \right)k^3 }} \cos kx } \\
&+& \sum\limits_{k = 1}^\infty {\frac{{\left( {1 - e^{
- \left( {1 + \varepsilon } \right)k^2 t} } \right)\left( {\varphi _{1k}^{\left( 3 \right)} - \psi
_{1k}^{\left( 3 \right)} } \right)}}{{\left( {1 - e^{ - \left( {1 + \varepsilon } \right)k^2 T} }\right)k^3 }}\sin kx},
\end{eqnarray*}
\begin{eqnarray*}
f\left( x \right) =
-\varphi'' \left( x \right) +\varepsilon \varphi'' \left( -x \right)&+&  \frac{\psi_0-\varphi_0}{T}
-\sum\limits_{k = 1}^\infty
{\frac{{\left( {1 - \varepsilon } \right)\left( {\varphi
_{2k}^{\left( 3 \right)}  - \psi _{2k}^{\left( 3 \right)} }
\right)}}{{\left( {1 - e^{ - \left( {1 - \varepsilon } \right)k^2
T} } \right)k }}\cos kx} \\
&+&\sum\limits_{k = 1}^\infty
{\frac{{\left( {1 + \varepsilon } \right)\left( {\varphi
_{1k}^{\left( 3 \right)}  - \psi _{1k}^{\left( 3 \right)} }
\right)}}{{\left( {1 - e^{ - \left( {1 + \varepsilon } \right)k^2
T} } \right)k }}\sin kx},
\end{eqnarray*}
where
$$\varphi _{0}  =
\frac{1}{2\pi}\int\limits_{ - \pi }^\pi  {\varphi \left( x \right)dx}, \quad \psi _{0}  =
\frac{1}{2\pi}\int\limits_{ - \pi }^\pi  {\psi \left( x \right)dx},$$
$$\varphi _{1k}^{(3)}  = \frac{1}{\pi}\int\limits_{ - \pi }^\pi  {\varphi''' \left( x \right)\cos kx \,dx}, \quad \varphi _{2k}^{(3)}  = \frac{1}{\pi}\int\limits_{ - \pi }^\pi  {\varphi''' \left( x \right)\sin kx \,dx},$$
$$\psi _{1k}^{(3)}  = \frac{1}{\pi}\int\limits_{ - \pi }^\pi {\psi''' \left( x \right)\cos kx \,dx},\quad \psi _{2k}^{(3)}  = \frac{1}{\pi}\int\limits_{ -\pi }^\pi  {\psi''' \left( x \right)\sin kx\,dx} .$$
\end{theorem}
\begin{theorem}\label{th4} Let $\varphi \left( x \right),\psi \left( x \right) \in C^2 \left[ { - \pi ,\pi } \right]$, $\varphi'''(x)$, $\psi'''(x) \in L_2(-\pi, \pi)$ and $\varphi
^{\left( i \right)} \left( { - \pi } \right) =  - \varphi ^{\left(
i \right)} \left( \pi  \right),$ $\psi ^{\left( i \right)} \left(
{ - \pi } \right) =  - \psi ^{\left( i \right)} \left( \pi
\right),$ $i = 0,1,2.$ Then, a unique solution to the Inverse Problem IP4
exists and it can be written in the form
\begin{eqnarray*}
u\left(
{x,t} \right) = \varphi \left( x \right) &-& \sum\limits_{k =
0}^\infty {\frac{{\left( {1 - e^{ - \left( {1 - \varepsilon }
\right)\left( {k + \frac{1}{2}} \right)^2 t} } \right) \left( {\varphi
_{2k}^{\left( 3 \right)} - \psi _{2k}^{\left( 3 \right)} }
\right)}}{{\left( {1 - e^{ - \left( {1 -
\varepsilon } \right)\left( {k + \frac{1}{2}} \right)^2 T} }
\right)\left( {k + \frac{1}{2}} \right)^3 }}\cos \left(
{k + \frac{1}{2}} \right)x}\\
&+& \sum\limits_{k = 0}^\infty  {\frac{{\left( {1
- e^{ - \left( {1 + \varepsilon } \right)\left( {k + \frac{1}{2}}
\right)^2 t} } \right) \left( {\varphi _{1k}^{\left( 3 \right)}  -
\psi _{1k}^{\left( 3 \right)} } \right)}}{{\left( {1 - e^{ - \left( {1 + \varepsilon }
\right)\left( {k + \frac{1}{2}} \right)^2 T} } \right)\left( {k +
\frac{1}{2}} \right)^3 }}\sin \left( {k + \frac{1}{2}}
\right)x},
\end{eqnarray*}

\begin{eqnarray*}
f\left( x
\right) = -\varphi'' \left( x \right) + \varepsilon \varphi'' \left( -x \right) &-& \sum\limits_{k = 0}^\infty {\frac{{\left( {1 - \varepsilon } \right)\left( {\varphi
_{2k}^{\left( 3 \right)}  - \psi _{2k}^{\left( 3 \right)} }
\right)}}{{\left( {k + \frac{1}{2}} \right)\left( {1 - e^{ - \left( {1 - \varepsilon }
\right)\left( {k + \frac{1}{2}} \right)^2 T} } \right)}}\cos \left( {k + \frac{1}{2}} \right)x} \\
&+& \sum\limits_{k = 0}^\infty  {\frac{{\left( {1 + \varepsilon }
\right)\left( {\varphi _{1k}^{\left( 3 \right)}  - \psi
_{1k}^{\left( 3 \right)} } \right)}}{{\left( {k + \frac{1}{2}} \right)\left( {1 - e^{ - \left( {1
+ \varepsilon } \right)\left( {k + \frac{1}{2}} \right)^2 T} }
\right)}}\sin \left( {k +
\frac{1}{2}} \right)x} ,
\end{eqnarray*}
where
$$\varphi _{1k}^{(2)}  = \frac{1}{\pi}\int\limits_{ -
\pi }^\pi  {\varphi''' \left( x \right)\cos \left( {k + \frac{1}{2}}
\right)x \,dx}, \quad \varphi _{2k}^{(2)}  = \frac{1}{\pi}\int\limits_{ - \pi }^\pi
{\varphi''' \left( x \right)\sin \left( {k + \frac{1}{2}} \right)x\,dx},$$
$$\psi _{1k}^{(2)}  = \frac{1}{\pi}\int\limits_{ - \pi }^\pi  {\psi''' \left( x\right)\cos \left( {k + \frac{1}{2}} \right)x\,dx},\quad \psi _{2k}^{(2)} =
\frac{1}{\pi}\int\limits_{ - \pi }^\pi  {\psi''' \left( x \right)\sin \left( {k +\frac{1}{2}} \right)x\,dx}.$$
\end{theorem}

\subsection{Example Solution}
For the sake of illustration, we present here a simple example solution for the Inverse Problem IP1. For this purpose, we consider the following choice of conditions (\ref{2.2}):
\begin{equation*}
u\left({x,0} \right) = 0, \quad u\left( {x,T} \right) = \sin x, \quad x \in \left[ { - \pi ,\pi } \right],
\end{equation*}
i.e., we have $\varphi \left( x \right)=0$ and $\psi \left( x \right)=\sin x$. Calculating the coefficients of the series solutions as given in Theorem \ref{th1}, we get
\[ u(x,t)= \frac{{{1-e^{ - \left( {1 + \varepsilon } \right) t}} }}{{{1 - e^{ - \left( {1 + \varepsilon } \right) T} } }} \, \sin x, \quad {\text{and}} \quad f(x)= \frac{{{1 + \epsilon}} }{{{1 - e^{ - \left( {1 + \varepsilon } \right) T} } }} \, \sin x. \]
These solutions are illustrated in the following figures:
\begin{figure}
\label{fig2}
\includegraphics[height=7cm,width=7cm]{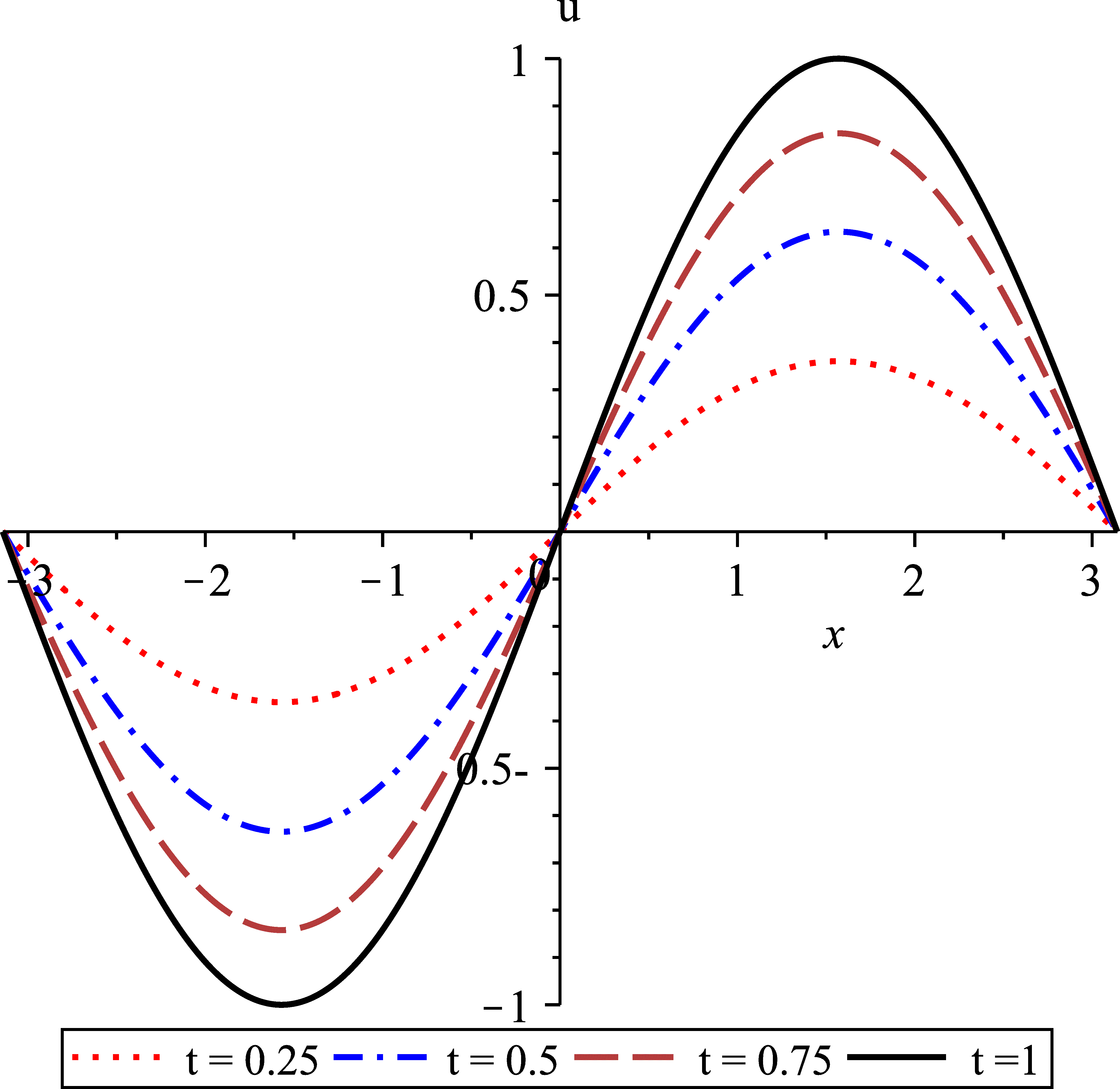}
\includegraphics[height=7.5cm,width=7cm]{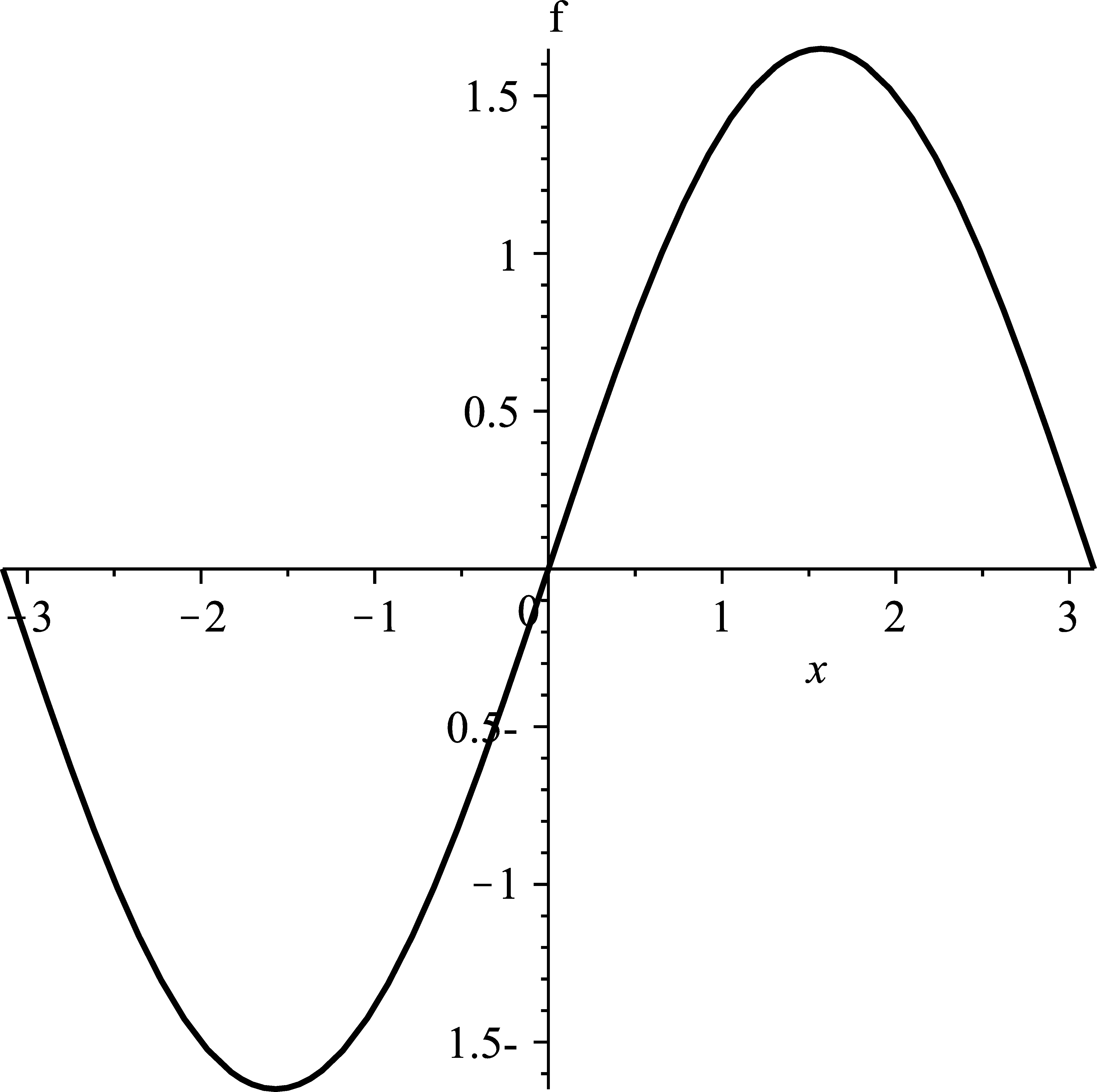}
\caption{Graphs of $u(x,t)$ at different times (left) and $f(x)$ (right) for $\epsilon=0.1$ and $T=1$.}
\end{figure}

\begin{figure}
\label{fig3}
\includegraphics[height=7cm,width=7cm]{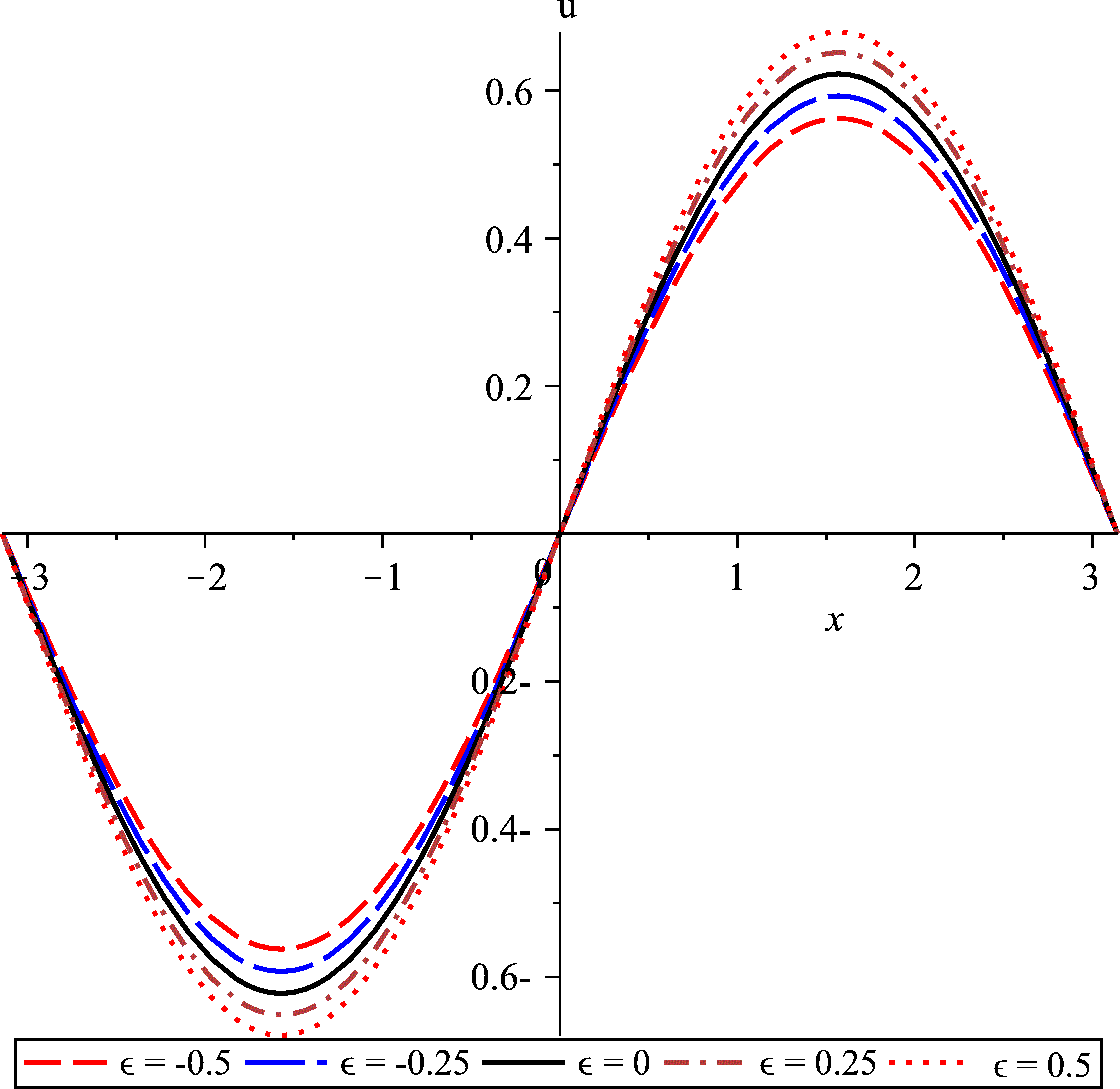}
\includegraphics[height=7cm,width=7cm]{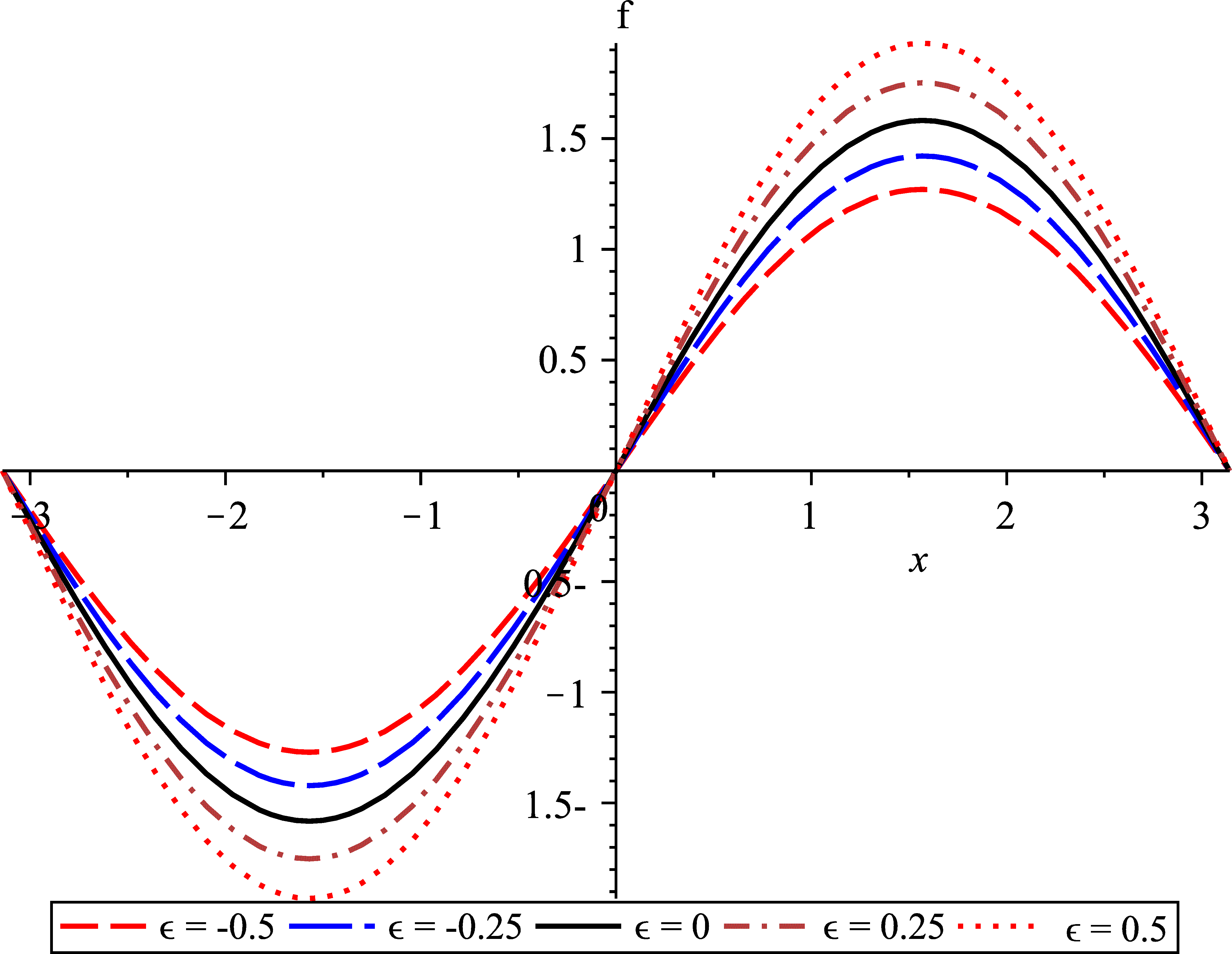}
\caption{Graphs of $u(x,t)$ at $t=0.5$ (left) and $f(x)$ (right) for different values of $\epsilon$ and for $T=1$.}
\end{figure}

\end{document}